\documentclass[12pt]{article}
\usepackage{times}
\usepackage{booktabs}
\usepackage{pifont}
\usepackage{caption}
\usepackage{mathrsfs}
\usepackage[fleqn]{amsmath}
\usepackage{amsfonts,amsthm,amssymb,mathrsfs,bbding}
\usepackage{txfonts}
\usepackage{graphics,multicol}
\usepackage{graphicx}
\usepackage{color}
\usepackage{enumerate}
\usepackage{caption}
\usepackage[
colorlinks,
anchorcolor=blue,
linkcolor=blue,
citecolor=blue]{hyperref}
\usepackage{extarrows}
\usepackage{cite}
\usepackage{latexsym,bm}
\usepackage{mathtools}
\evensidemargin=\oddsidemargin\topmargin=-1.5cm

\newtheorem{thm}{Theorem}[section]
\newtheorem{prob}{Problem}
\newtheorem{lem}[thm]{Lemma}
\newtheorem{cor}[thm]{Corollary}

\newtheorem{remark}{Remark}

\theoremstyle{definition}

\addtocounter{section}{0}
\begin{document}
\title{Automorphism group of the complete alternating group graph\footnote{Supported
by the National Natural Science Foundation of China (Grant Nos. 11671344, 11531011).}}
\author{{\small Xueyi Huang, \ \ Qiongxiang Huang\footnote{
Corresponding author.}\setcounter{footnote}{-1}\footnote{
\emph{E-mail address:} huangxymath@gmail.com (X.Y. Huang), huangqx@xju.edu.cn (Q.X. Huang).}}\\[2mm]\scriptsize
College of Mathematics and Systems Science,
\scriptsize Xinjiang University, Urumqi, Xinjiang 830046, P. R. China}
\date{}
\maketitle
{\flushleft\large\bf Abstract}
Let $S_n$ and $A_n$ denote the symmetric group and alternating group of degree $n$ with $n\geq 3$, respectively. Let $S$ be the set  of all $3$-cycles in $S_n$. The \emph{complete alternating group graph}, denoted by $CAG_n$, is defined as the Cayley graph $\mathrm{Cay}(A_n,S)$ on $A_n$ with respect  to $S$. In this paper,  we show that $CAG_n$ ($n\geq 4$) is not a normal Cayley graph. Furthermore,  the automorphism group of $CAG_n$ for $n\geq 5$ is obtained, which  equals to $\mathrm{Aut}(CAG_n)=(R(A_n)\rtimes \mathrm{Inn}(S_n))\rtimes \mathbb{Z}_2\cong (A_n\rtimes S_n)\rtimes \mathbb{Z}_2$, where $R(A_n)$ is the right regular representation of $A_n$, $\mathrm{Inn}(S_n)$ is the inner automorphism group of $S_n$, and $\mathbb{Z}_2=\langle h\rangle$, where $h$ is the map $\alpha\mapsto\alpha^{-1}$ ($\forall \alpha\in A_n$).
\begin{flushleft}
\end{flushleft}

\section{Introduction}\label{s-1}
Let $X=(V,E)$ be a simple undirected graph. An automorphism of $X$ is a permutation on its vertex set $V$ that preserves adjacency relations. The \emph{automorphism group} of $X$, denoted by $\mathrm{Aut(X)}$, is the set of all automorphisms of $X$.

For a finite group $\Gamma$, and a subset $T$ of $\Gamma$ such that $e\not\in T$ ($e$ is the identity element of $\Gamma$) and  $T=T^{-1}$, the \emph{Cayley graph} $\mathrm{Cay}(\Gamma,T)$ on $\Gamma$ with respect to $T$ is defined as the undirected graph with vertex set $\Gamma$ and  edge set $\{(\gamma,t\gamma )\mid\gamma\in\Gamma,t\in T\}$. The \emph{right regular representation} $R(\Gamma)=\{r_\gamma:x\mapsto x\gamma~(\forall x\in \Gamma)\mid\gamma\in\Gamma\}$, i.e., the action of $\Gamma$ on itself by right multiplication,  is a subgroup of the automorphism group $\mathrm{Aut}(\mathrm{Cay}(\Gamma,T))$ of the Cayley graph $\mathrm{Cay}(\Gamma,T)$.  Hence, every Cayley graph is vertex-transitive. Furthermore, the group $\mathrm{Aut}(\Gamma, T)=\{\sigma\in \mathrm{Aut}(\Gamma)\mid T^\sigma=T\}$ is  a
subgroup of $\mathrm{Aut}(\mathrm{Cay}(\Gamma,T))_e$, the
stabilizer of the identity vertex $e$ in $\mathrm{Aut}(\mathrm{Cay}(\Gamma,T))$, and so is also a subgroup of  $\mathrm{Aut}(\mathrm{Cay}(\Gamma,T))$. The Cayley graph $\mathrm{Cay}(\Gamma,T)$ is said to be
\emph{normal} if $R(\Gamma)$ is a normal subgroup of $\mathrm{Aut}(\mathrm{Cay}(\Gamma,T))$. By Godsil \cite{Godsil},  $N_{\mathrm{Aut}(\mathrm{Cay}(\Gamma,T))}(R(\Gamma)) = R(\Gamma)\rtimes \mathrm{Aut}(\Gamma, T)$. Thus, $\mathrm{Cay}(\Gamma,T)$ is normal if and only if $\mathrm{Aut}(\mathrm{Cay}(\Gamma,T))= R(\Gamma)\rtimes \mathrm{Aut}(\Gamma, T)$.

A basic problem in algebraic graph theory is to determine the (full) automorphism groups of Cayley graphs. As the (full) automorphism group of a  normal Cayley graph $\mathrm{Cay}(\Gamma,T)$  is always equal to  $\mathrm{Aut}(\mathrm{Cay}(\Gamma,T))= R(\Gamma)\rtimes \mathrm{Aut}(\Gamma, T)$, to determine  the normality of  Cayley graphs  is an important problem in the literature. The whole information about the normality of Cayley graphs on the cyclic groups of prime order and the groups of order twice a prime were gained by Alspach \cite{Alspach} and Du et al. \cite{Du}, respectively.   Wang et al. \cite{Wang} obtained all disconnected normal Cayley graphs. Fang et al. \cite{Fang} proved that the vast majority of connected cubic Cayley graphs on non-abelian simple groups are normal. Baik et al. \cite{Baik,Baik1} listed all connected non-normal Cayley graphs on abelian groups with valency less than $6$ and Feng et al. \cite{Feng} proved that all connected tetravalent Cayley graphs on $p$-groups of nilpotent class $2$ with $p$ an odd prime are normal. For more results regarding automorphism groups and normality of Cayley graphs, we refer the reader to \cite{Deng1,Feng2,Feng3,Godsil,Godsil2,Huang,Li,Li1,Praeger,Xu,Zhou1}.

With regard to the symmetric group $S_n$ ($n\geq 3$), let $T$ be  a set of transpositions generating $S_n$. The \emph{transposition graph} $\mathrm{Tra}(T)$ of $T$  is defined as the graph with vertex set $\{1,2,\ldots,n\}$ and with an edge connecting two vertices $i$ and $j$ iff $(i,j)\in T$. A set of transposition $T$ can generate $S_n$ iff  $\mathrm{Tra}(T)$ is connected. Godsil and Royle \cite{Godsil1} showed that if $\mathrm{Tra}(T)$ is an asymmetric tree, then  $\mathrm{Aut}(\mathrm{Cay}(S_n,T))=R(S_n)\cong S_n$, where $R(S_n)$ is the right regular representation of $S_n$. Feng \cite{Feng1} showed that if $\mathrm{Tra}(T)$ is an arbitrary tree, then $\mathrm{Aut}(\mathrm{Cay}(S_n,T))=R(S_n)\rtimes \mathrm{Aut}(S_n,T)$. Ganesan \cite{Ganesan} showed that if the girth of $\mathrm{Tra}(T)$ is at least $5$, then $\mathrm{Aut}(\mathrm{Cay}(S_n,T))=R(S_n)\rtimes \mathrm{Aut}(S_n,T)$. Thus, all these Cayley graphs are normal. However, Ganesan \cite{Ganesan} showed that if $\mathrm{Tra}(T)$ is a $4$-cycle graph, then $\mathrm{Cay}(S_n,T)$ is non-normal. For non-normal Cayley graphs, it seems a difficult work to obtain the full automorphism groups.  Ganesan \cite{Ganesan1} showed that if $\mathrm{Tra}(T)$ is a complete graph, then $\mathrm{Cay}(S_n,T)$ is non-normal and $\mathrm{Aut}(\mathrm{Cay}(S_n,T))=(R(S_n)\rtimes \mathrm{Aut}(S_n,T))\rtimes\mathbb{Z}_2=(R(S_n)\rtimes \mathrm{Inn}(S_n))\rtimes\mathbb{Z}_2$, where $\mathrm{Inn}(S_n)$ is the inner automorphism group of $S_n$, and $\mathbb{Z}_2=\langle h\rangle$, where $h$ is the map $\alpha\mapsto\alpha^{-1}$ ($\forall \alpha\in S_n$). Let  $\Gamma_n$ be the \emph{derangement graph} which is defined to be  the Cayley graph $\mathrm{Cay}(S_n,D)$, where $D$ is the set of fixed point free permutations of $S_n$. Also, Deng et al. \cite{Deng} showed that  $\mathrm{Aut}(\Gamma_n)=(R(S_n)\rtimes \mathrm{Inn}(S_n))\rtimes\mathbb{Z}_2$, and hence $\Gamma_n$ is non-normal.

With regard to the alternating group $A_n$ ($n\geq 3$), set $T=\{(1,2,i),(1,i,2)\mid 3\leq i\leq n\}$. The \emph{alternating group graph}, denoted by $AG_n$, is defined as the Cayley graph $\mathrm{Cay}(A_n,T)$. Since $T$ can generate $A_n$ (cf. \cite{Suzuki}, p.298),  $AG_n$ is connected. The alternating group graph was introduced by Jwo et al. \cite{Jwo} as an interconnection network topology for computing systems. Zhou \cite{Zhou} showed that $\mathrm{Aut}(AG_n)=R(A_n)\rtimes \mathrm{Aut}(A_n,T)\cong A_n\rtimes (S_{n-2}\times S_2)$, and hence $AG_n$ is normal.  Motivated by the work of Zhou \cite{Zhou} and Ganesan \cite{Ganesan}, we define the \emph{complete alternating group graph} $CAG_n$ as the Cayley graph $\mathrm{Cay}(A_n,S)$ on $A_n$ ($n\geq 3$) with respect to $S=\{(i,j,k),(i,k,j)\mid 1\leq i<j<k\leq n\}$, the complete set of $3$-cycles in $S_n$. Clearly, $CAG_n$ is  a connected Cayley graph. Furthermore,  $CAG_n$ can be viewed as one of the two isomorphic connected components of the $(n-3)$-point fixing graph $\mathcal{F}(n,n-3)$ defined in \cite{Cheng}, which has only integral eigenvalues (cf. \cite{Cheng}, Corollary 1.2). If $n=3$, $CAG_n=K_3$ is obviously normal and $\mathrm{Aut}(CAG_3)\cong S_3$.

In the present paper,  it is shown that $CAG_n$ ($n\geq 4$) is not a normal Cayley graph. Moreover, the automorphism group of $CAG_n$ ($n\geq 5$) is shown to equal $\mathrm{Aut}(CAG_n)=(R(A_n)\rtimes \mathrm{Aut}(S_n,S))\rtimes \mathbb{Z}_2=(R(A_n)\rtimes \mathrm{Inn}(S_n))\rtimes \mathbb{Z}_2\cong (A_n\rtimes S_n)\rtimes \mathbb{Z}_2$, where $R(A_n)$ is the right regular representation of $A_n$, $\mathrm{Inn}(S_n)$ is the inner automorphism group of $S_n$, and $\mathbb{Z}_2=\langle h\rangle$, where $h$ is the map $\alpha\mapsto\alpha^{-1}$ ($\forall \alpha\in A_n$). It follows that $CAG_n$ is arc-transitive.

\section{Non-normality of the complete alternating group graph}\label{s-2}
In this section, we list some important results in group theory and algebraic graph theory which are useful throughout this paper, and show that the complete alternating group graph $CAG_n$ is non-normal for all $n\geq 4$.

\begin{lem}[\cite{Suzuki}, Chapter 3, Theorems 2.17--2.20]\label{lem-2-1}
If $n \geq 4$ and $n\neq 6$, then $\mathrm{Aut}(S_n)=\mathrm{Aut}(A_n)=\mathrm{Inn}(S_n)\cong S_n$.  If $n = 6$, then $|\mathrm{Aut}(S_n):\mathrm{Inn}(S_n)|=|\mathrm{Aut}(A_n):\mathrm{Inn}(S_n)|=2$, and each element in $\mathrm{Aut}(S_n)\setminus\mathrm{Inn}(S_n)$ maps a transposition to a product of three disjoint transpositions, each element in $\mathrm{Aut}(A_n)\setminus\mathrm{Inn}(S_n)$ maps a $3$-cycle to a product of two disjoint $3$-cycles.
\end{lem}

The \emph{Kneser graph}, denoted by $KG(n,r)$, is the graph with the $r$-subsets of a fixed $n$-set as its vertices, with two $r$-subsets adjacent if they are disjoint. Clearly, $KG(n,1)$ is the complete graph, and $KG(n,r)$ is the empty graph if $n<2r$. The following lemma shows that the automorphism group of $KG(n,r)$ ($n\geq 2r+1$) is isomorphic to the symmetric group $S_n$.

\begin{lem}[\cite{Godsil1}, Chapter 7, Corollary 7.8.2]\label{lem-2-2}
If $n\geq 2r+1$, then $\mathrm{Aut}(KG(n,r))\cong S_{n}$.
\end{lem}

The following lemma gives a criterion for  Cayley graphs to be normal.

\begin{lem}[\cite{Xu}]\label{lem-2-3}
Let $X=\mathrm{Cay}(\Gamma,T)$ be the Cayley graph on $\Gamma$ with respect to $T$. Let $\mathrm{Aut}(X)_e$ be the set of automorphisms of $X$ that fixes the identity vertex $e$. Then $X$ is normal
if and only if every element of $\mathrm{Aut}(X)_e$ is an automorphism of the group $\Gamma$.
\end{lem}

Using Lemma \ref{lem-2-3}, we now prove that $CAG_n$ ($n\geq 4$) is not a normal Cayley graph.

\begin{thm}\label{thm-2-1}
Let $S$ be the complete set of $3$-cycles in $S_n$ ($n\geq 4$). Then, the map $h: \alpha\mapsto\alpha^{-1}$ ($\forall \alpha\in A_n$) is an automorphism of the complete alternating group graph $CAG_n=\mathrm{Cay}(A_n,S)$. In
particular, $CAG_n$ is non-normal.
\end{thm}
\begin{proof}
Note that $S$ consists of all $3$-cycles in $S_n$. Thus, to show that the map $h: \alpha\mapsto\alpha^{-1}$ ($\forall \alpha\in A_n$) is an automorphism of  $CAG_n$, it suffices to show that for any $\alpha,\beta\in A_n$, if $\alpha\beta^{-1}$ is a $3$-cycle, then $\alpha^{-1}\beta$ is a $3$-cycle. In fact, putting $\alpha\beta^{-1}=(i,j,k)$, we have $\alpha^{-1}=\beta^{-1}(i,k,j)$, and so $\alpha^{-1}\beta=\beta^{-1}(i,k,j)\beta$ is a $3$-cycle.

Now we shall prove that $CAG_n$ is non-normal. Let $G:=\mathrm{Aut}(CAG_n)$, and let $G_e$ denote the stabilizer of the identity vertex $e$ in $G$. Clearly, $h\in G_e$. However, $h$ is not an automorphism of the group $A_n$ because $A_n$ is non-abelian for $n\geq 4$. Hence, by Lemma \ref{lem-2-3}, the Cayley graph $CAG_n=\mathrm{Cay}(A_n,S)$ is non-normal.

The proof is now complete.
\end{proof}
\begin{remark}\label{rem-1}
\emph{It is well known that $A_n$ is a simple group for $n\geq 5$, and $S_n=A_n\rtimes S_2$. Thus the only normal subgroups of $S_n$ are the trivial group, the alternating group $A_n$ and the symmetric group $S_n$ itself. For any fixed $k$ ($4\leq k\leq n$), all $k$-cycles in $S_n$ form a full conjugacy class, it follows that the subgroup they generate is normal. This implies that all $k$-cycles in $S_n$ generate $A_n$ (resp. $S_n$) if $k$ is odd (resp. even). Let $T$  be the set of all $k$-cycles ($k\geq 4$) in $S_n$. As in the proof of Theorem \ref{thm-2-1}, from Lemma \ref{lem-2-3} one can easily deduce that  the Cayley graph $\mathrm{Cay}(\Gamma,T)$ is non-normal, where $\Gamma=A_n$ if $k$ is odd, and $\Gamma=S_n$ if $k$ is even.}
\end{remark}

\section{Automorphism group of the complete alternating group graph}\label{s-3}
In this section, we focus on determining the full automorphism group  of the complete alternating group graph $CAG_n$. First of all, the following theorem suggests that  $\mathrm{Aut}(CAG_n)$  contains $(R(A_n)\rtimes \mathrm{Inn}(S_n))\rtimes \mathbb{Z}_2\cong (A_n\rtimes S_n)\rtimes\mathbb{Z}_2$ as its subgroup, and so is of order at least $(n!)^2$, for all $n\geq 4$.

\begin{thm}\label{thm-2-2}
Let $S$ be the complete set of $3$-cycles in $S_n$ ($n\geq 4$) and let $CAG_n=Cay(A_n,S)$ be the complete alternating group graph. Then
$$\mathrm{Aut}(CAG_n)\supseteq (R(A_n)\rtimes \mathrm{Inn}(S_n))\rtimes \mathbb{Z}_2\cong (A_n\rtimes S_n)\rtimes \mathbb{Z}_2,$$
where $R(A_n)$ is the right regular representation of $A_n$, $\mathrm{Inn}(S_n)$ is the inner automorphism group of $S_n$, and $\mathbb{Z}_2=\langle h\rangle$, and $h$ is the map $\alpha\mapsto\alpha^{-1}$ ($\forall \alpha\in A_n$).
\end{thm}
\begin{proof}
Let $G:=\mathrm{Aut}(CAG_n)$. Since both  $R(A_n)$ and $\mathrm{Aut}(A_n,S)$ are subgroups of $G$, we have $G\supseteq N_{G}(R(A_n))=R(A_n)\rtimes \mathrm{Aut}(A_n,S)$ (cf. \cite{Godsil}). For any $g\in \mathrm{Aut}(A_n,S)$, by Lemma \ref{lem-2-1} we may conclude that $g\in \mathrm{Inn}(S_n)$ because $g$ maps $3$-cycles to $3$-cycles. Thus $\mathrm{Aut}(A_n,S)\subseteq \mathrm{Inn}(S_n)$. Furthermore, we claim that each $g\in \mathrm{Inn}(S_n)$ ($\leq\mathrm{Aut}(A_n)$) fixes $S$ setwise because $S$ contains all $3$-cycles in $S_n$ and $g$ maps $3$-cycles to $3$-cycles, and hence $g\in \mathrm{Aut}(A_n,S)$. Therefore, $\mathrm{Aut}(A_n,S)=\mathrm{Inn}(S_n)\cong S_n$.

The map $h:\alpha\mapsto\alpha^{-1}$ ($\forall\alpha\in A_n$) fixes the identity vertex $e$ of $CAG_n$ and hence fixes its neighborhood $S$ setwise. We claim that $h\not\in R(A_n)\rtimes \mathrm{Inn}(S_n)$. Since otherwise $h=r_\beta\sigma_\gamma$, where $r_\beta\in R(A_n)$ and $\sigma_\gamma\in \mathrm{Inn}(S_n)$. This means that $r_\beta=r_e$ because both $h$ and $\sigma_\gamma$ fix the identity vertex $e$, and so $h=\sigma_\beta\in \mathrm{Inn}(S_n)=\mathrm{Aut}(A_n,S)$, which is impossible. Thus $G$ contains $H:=(R(A_n)\rtimes \mathrm{Inn}(S_n))\rtimes \mathbb{Z}_2$ as its subgroup, where $\mathbb{Z}_2=\langle h\rangle$, and $R(A_n)\rtimes \mathrm{Inn}(S_n)$ has index $2$ in $H$ and hence is a normal subgroup of $H$.

The proof is now complete.
\end{proof}
\begin{remark}\label{rem-2}
\emph{As  $\mathrm{Aut}(A_n,S)=\mathrm{Inn}(S_n)\leq (\mathrm{Aut}(CAG_n))_e$ acts transitively on the neighborhood $S$ of the identity vertex $e$, $CAG_n$ is arc-transitive because it is vertex-transitive.}
\end{remark}

\begin{remark}\label{rem-3}
\emph{Let $n\geq 5$, $n\neq 6$ and $4\leq k\leq n$, and let $T$ be the set of all $k$-cycles in $S_n$. As in the proof of Theorem \ref{thm-2-2}, from  Remark \ref{rem-1} and Lemma \ref{lem-2-3} one can obtain that  $\mathrm{Aut}(\mathrm{Cay}(\Gamma,T))\supseteq (R(\Gamma)\rtimes \mathrm{Inn}(S_n))\rtimes \mathbb{Z}_2\cong (\Gamma\rtimes S_n)\rtimes \mathbb{Z}_2$,  where $\Gamma=A_n$ if $k$ is odd, and $\Gamma=S_n$ if $k$ is even, and $\mathbb{Z}_2=\langle h\rangle$, where $h$ is the map $\alpha\mapsto\alpha^{-1}$ ($\forall \alpha\in \Gamma$).}
\end{remark}

In what follows, we will prove that $\mathrm{Aut}(CAG_n)$ has order at most $(n!)^2$ for $n\geq 5$, and so is exactly equal to the subgroup  given in Theorem \ref{thm-2-2}. The proof consists of a series of lemmas, so we present a figure (see Fig. \ref{fig-1}) to show that how step by step we achieve ultimately this goal and how those lemmas are related.

Recall that the symmetric group $S_n$ (resp. the alternating group $A_n$) can be viewed as the group of all (resp. even) permutations on the set $\{1,2,\ldots,n\}$. The \emph{support} of a permutation $\tau$, denoted by $\mathrm{supp}(\tau)$, is the set of points moved by $\tau$. The following lemma is critical to our main result.

\begin{figure}[h]
\centering
\small
\unitlength 2mm 
\linethickness{0.4pt}
\ifx\plotpoint\undefined\newsavebox{\plotpoint}\fi 
\begin{picture}(45,33)(0,0)
\put(0,30){\makebox(0,0)[cc]{Lemma \ref{lem-3-1}}}
\put(0,26){\makebox(0,0)[cc]{$\Downarrow$}}
\put(0,22){\makebox(0,0)[cc]{$A_e=\{1\}$ (Lemma \ref{lem-3-2})}}
\put(0,18){\makebox(0,0)[cc]{$\Downarrow$}}
\put(0,13){\makebox(0,0)[cc]{\begin{minipage}{5cm}
$G_e=G_e/A_e$ can be viewed as a permutation group on $S$.
\end{minipage}}}
\put(0,8){\makebox(0,0)[cc]{$\Downarrow$}}
\put(6.0,8){\makebox(0,0)[cc]{Lemma \ref{lem-3-3}}}
\put(0,1){\makebox(0,0)[cc]{\begin{minipage}{5cm}$G_e$ acts on $\Delta$, and $G_e/K$ can be viewed as a permutation group on $\Delta$.\end{minipage}}}
\put(22,1){\makebox(0,0)[cc]{$\xLongrightarrow[\mbox{Lemma~\ref{lem-3-4}}]{\mbox{Lemma~\ref{lem-3-5}}}$}}
\put(44,1){\makebox(0,0)[cc]{$G_e/K\leq\mathrm{Aut}(L_n)\cong S_n$}}
\put(44,8){\makebox(0,0)[cc]{$\Uparrow$}}
\put(50,8){\makebox(0,0)[cc]{Lemma \ref{lem-3-6}}}
\put(44,13){\makebox(0,0)[cc]{$G_e=|G_e/K|\cdot|K|\leq|S_n|\cdot|K|\leq 2n!$}}
\put(44,18){\makebox(0,0)[cc]{$\Uparrow$}}
\put(44,22){\makebox(0,0)[cc]{$|\mathrm{Aut}(CAG_n)|=|R(A_n)|\cdot|G_e|\leq (n!)^2$}}
\put(44,26){\makebox(0,0)[cc]{$\Uparrow$}}
\put(50.5,26){\makebox(0,0)[cc]{Theorem \ref{thm-2-2}}}
\put(44,30){\makebox(0,0)[cc]{\begin{minipage}{5.5cm}$\mathrm{Aut}(CAG_n)=(R(A_n)\rtimes \mathrm{Inn}(S_n))\rtimes \mathbb{Z}_2\cong (A_n\rtimes S_n)\rtimes \mathbb{Z}_2$\end{minipage}}}
\end{picture}
\vspace{0.3cm}
\caption{\small{Illustrating the proof of the main result.}\label{fig-1}}
\end{figure}
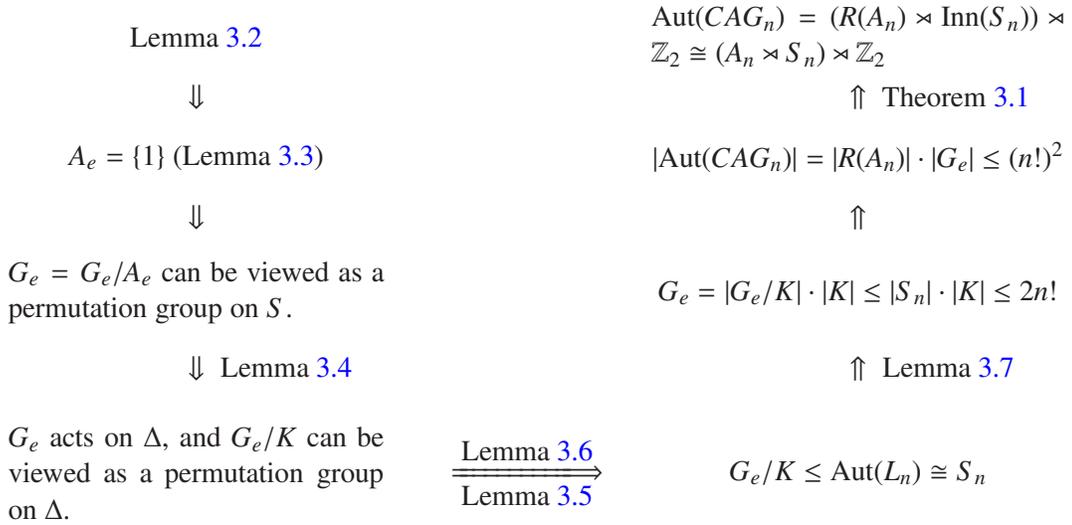

\begin{lem}\label{lem-3-1}
Let $S$ be the complete set of $3$-cycles in $S_n$ ($n\geq 5$). Let $\tau,\kappa\in S$, $\tau\neq \kappa$. Then $\tau\kappa=\kappa\tau\neq e$ if and only if there is a unique $4$-cycle in $CAG_n$ containing $e$, $\tau$ and $\kappa$.
\end{lem}
\begin{proof}
Suppose $\tau\kappa=\kappa\tau\neq e$. Then $\tau^{-1}\kappa \tau=\kappa$. It follows that $\mathrm{supp}(\tau)\cap\mathrm{supp}(\kappa)=\emptyset$. Clearly, $\tau\kappa=\kappa\tau$ is a common neighbor of the vertices $\tau$ and $\kappa$ other than $e$. Let $\omega$ be another  common neighbor of them. There exists $x,y\in S$ such that $x\tau=y\kappa=\omega$. Note that $x\tau=y\kappa$ iff $\tau\kappa^{-1}=x^{-1}y$. Since $\mathrm{supp}(\tau)\cap\mathrm{supp}(\kappa^{-1})=\emptyset$, $\tau\kappa^{-1}=x^{-1}y$ iff $\tau=x^{-1}$ and $\kappa^{-1}=y$ or $\tau=y$ and $\kappa^{-1}=x^{-1}$, i.e., $x=\tau^{-1}$ and $y=\kappa^{-1}$ or $x=\kappa$ and $y=\tau$. Thus $\omega=e$ or $\omega=\tau\kappa$. Hence, there exists a unique $4$-cycle in $CAG_n$ containing $e$, $\tau$ and $k$, namely the cycle $(e,\tau,\tau\kappa=\kappa\tau,k,e)$.

To prove the converse, suppose $\tau\kappa\neq\kappa\tau$ or $\tau\kappa=\kappa\tau=e$. First we suppose that $\tau\kappa\neq\kappa\tau$. Then $|\mathrm{supp}(\tau)\cap\mathrm{supp}(\kappa)|=1$ or $2$.  We  consider the following two cases:

\noindent{\bf Case 1.}  $|\mathrm{supp}(\tau)\cap\mathrm{supp}(\kappa)|=1$;

Without loss of generality, let  $\tau=(1,2,3)$ and $\kappa=(1,4,5)$, and let $\omega$ be a common neighbor of $\tau$ and $\kappa$. Then $\omega=x\tau=y\kappa$ for some $x,y\in S$. As before, $x\tau=y\kappa$ iff $x^{-1}y=\tau\kappa^{-1}=(1,2,3)(1,5,4)=(1,2,3,5,4)$. Observe that if we decompose $(1,2,3,5,4)$ as the product of two $3$-cycles,  the supports of these two $3$-cycles must lie in $\{1,2,3,4,5\}$ and contain exactly one common point. Using this fact, by simple computation we find that the only ways to decompose $(1,2,3,5,4)$ as a product of two $3$-cycles are
\begin{align*}
(1,2,3,5,4)&=(1,2,3)(1,5,4)=(2,3,5)(2,4,1)=(3,5,4)(3,1,2)\\
&=(5,4,1)(5,2,3)=(4,1,2)(4,3,5).
\end{align*}
Thus we must have $x=(1,3,2)$ and  $y=(1,5,4)$, $x=(2, 5,3)$ and  $y=(2,4,1)$, $x=(3,4,5)$ and  $y=(3,1,2)$, $x=(5,1,4)$ and  $y=(5,2,3)$ or $x=(4,2,1)$ and  $y=(4,3,5)$. Hence, we have
\begin{equation}\label{equ-1}
\omega\in N((1,2,3))\cap N((1,4,5))=\{e,(1,2,5),(1,2,3,4,5),(1,4,5,2,3),(1,4,3)\}.
\end{equation}
Therefore, there are exactly four $4$-cycles containing $e$, $\tau$ and $\kappa$.

\noindent{\bf Case 2.} $|\mathrm{supp}(\tau)\cap\mathrm{supp}(\kappa)|=2$.

Without loss of generality, we just need to consider the following two situations.

{\bf Subcase 2.1.} $\tau=(1,2,3)$ and $\kappa=(1,2,4)$;

Let $\omega$ be a common neighbor of $\tau$ and $\kappa$. We have $\omega=x\tau=y\kappa$ for some $x,y\in S$. As before, $x\tau=y\kappa$ iff $x^{-1}y=\tau\kappa^{-1}=(1,2,3)(1,4,2)=(2,3,4)$. Observe that if we decompose $(2,3,4)$ as the product of two $3$-cycles,  the supports of these two $3$-cycles must contain at least two common points. Using this fact, by simple computation we see that the only ways to decompose $(2,3,4)$ as a product of two $3$-cycles are
\begin{align*}
(2,3,4)&=(1,2,3)(1,4,2)=(1,4,2)(1,3,4)=(1,3,4)(1,2,3)=(2,4,3)(2,4,3)\\
&=(i,2,3)(i,4,2)=(i,4,2)(i,3,4)=(i,3,4)(i,2,3),
\end{align*}
where $5\leq i\leq n$. Thus we must have $x=(1,3,2)$ and $y=(1,4,2)$, $x=(1,2,4)$ and $y=(1,3,4)$,  $x=(1,4,3)$ and $y=(1,2,3)$,  $x=(2,3,4)$ and $y=(2,4,3)$,  $x=(i,3,2)$ and $y=(i,4,2)$,  $x=(i,2,4)$ and $y=(i,3,4)$ or $x=(i,4,3)$ and $y=(i,2,3)$, whence $\omega=e$, $(1,3)(2,4)$, $(1,4)(2,3)$, $(1,2)(3,4)$, $(1,2,i)$, $(1,2,4,i,3)$ or $(1,2,3,i,4)$ for $5\leq i\leq n$. This implies that $\tau$ and $\kappa$ have exactly $3n-8$ common neighbors (containing $e$). Therefore, the number of $4$-cycles containing $e$, $\tau$ and $\kappa$ is equal to $3n-9$, which is greater than $1$ due to $n\geq 5$.

{\bf Subcase 2.2.} $\tau=(1,2,3)$ and $\kappa=(1,4,2)$.

Let $\omega$ be a common neighbor of $\tau$ and $\kappa$. We have $\omega=x\tau=y\kappa$ for some $x,y\in S$. As before, $x\tau=y\kappa$ iff $x^{-1}y=\tau\kappa^{-1}=(1,2,3)(1,2,4)=(1,4)(2,3)$. Observe that if we decompose $(1,4)(2,3)$ as the product of two $3$-cycles,  the supports of these two $3$-cycles must contain exactly two common points, and these two common points must belong to $\{1,2,3,4\}$. Using this fact, by simple computation we see that the only ways to decompose $(1,4)(2,3)$ as a product of two $3$-cycles are
\begin{align*}
(1,4)(2,3)&=(1,2,3)(1,2,4)=(1,3,2)(1,3,4)=(1,4,2)(1,3,2)=(1,2,4)(2,4,3)\\
&=(1,4,3)(1,2,3)=(1,3,4)(2,3,4)=(2,4,3)(1,4,3)=(2,3,4)(1,4,2).
\end{align*}
Thus we must have $x=(1,3,2)$ and $y=(1,2,4)$,  $x=(1,2,3)$ and $y=(1,3,4)$,  $x=(1,2,4)$ and $y=(1,3,2)$, $x=(1,4,2)$ and $y=(2,4,3)$, $x=(1,3,4)$ and $y=(1,2,3)$,  $x=(1,4,3)$ and $y=(2,3,4)$, $x=(2,3,4)$ and $y=(1,4,3)$ or $x=(2,4,3)$ and $y=(1,4,2)$. Hence, we have
\begin{align}\label{equ-2}
\omega&\in N((1,2,3))\cap N((1,4,2))\nonumber\\
&=\{e,(1,3,2),  (1,3)(2,4),  (1,4,3), (2,3,4), (1,4)(2,3), (1,2)(3,4), (1,2,4)\}.
\end{align}
This implies that $\tau$ and $\kappa$ have exactly eight common neighbors (containing $e$). Therefore, there are exactly seven $4$-cycles containing $e$, $\tau$ and $\kappa$.

Next we suppose that $\tau\kappa=\kappa\tau=e$. Then $\tau=\kappa^{-1}$. Without loss of generality, let $\tau=(1,2,3)$ and  $\kappa=(1,3,2)$, and let $\omega$ be a common neighbor of $\tau$ and $\kappa$. We have $\omega=x\tau=y\kappa$ for some $x,y\in S$. As before, $x\tau=y\kappa$ iff $x^{-1}y=\tau\kappa^{-1}=(1,2,3)(1,2,3)=(1,3,2)$. As in Subcase 2.1, we see that the only ways to decompose $(1,3,2)$ as a product of two $3$-cycles are \begin{align*}
(1,3,2)=(1,2,3)(1,2,3)=(i, 1,3)(i,2, 1)=(i,3, 2)(i, 1,3)=(i,2, 1)(i,3, 2),
\end{align*}
where $4\leq i\leq n$. Thus we  have $x=(1,3,2)$ and  $y=(1,2,3)$, $x=(i,3,1)$ and $y=(i,2,1)$, $x=(i,2, 3)$ and $y=(i,1,3)$ or $x=(i,1,2)$ and $y=(i,3, 2)$, whence $\omega=e,(1,i)(2,3),(1,2)(3,i)$ or $(1,3)(2,i)$ for  $4\leq i\leq n$. This implies that $\tau$ and $\kappa$ have exactly $3n-8$ common neighbors (containing $e$). Therefore, the number of $4$-cycles containing $e$, $\tau$ and $\kappa$ is equal to $3n-9$, which is greater than $1$ due to $n\geq 5$.

The proof is now complete.
\end{proof}

\begin{remark}\label{rem-4}
\emph{In fact, as in  the proof of Lemma \ref{lem-3-1}, we have obtained  the common neighbors of any two vertices in $S$ (see Tab. \ref{tab-1}).}
\end{remark}
\begin{table}[t]
\caption{\label{tab-1}\small{Common neighbors of the vertices in $S$.}}
\resizebox{!}{1.21cm}{
\begin{tabular*}{23.0cm}{@{\extracolsep{\fill}}ccccc}
\toprule
$\tau$ &$\kappa$&$|\mathrm{supp}(\tau)\cap \mathrm{supp}(\kappa)|$&$N(\tau)\cap N(\kappa)$&$|N(\tau)\cap N(\kappa)|$\\
 \midrule
$(i,j,k)$ & $(p,q,r)$ &$0$& $\{e,(i,j,k)(p,q,r)\}$ & $2$\\
$(i,j,k)$ & $(i,p,q)$ &$1$& $\{e,(i,p,k),(i,p,q,j,k),(i,j,q),(i,j,k,p,q)$\} & $5$\\
$(i,j,k)$ & $(i,j,l)$ &$2$& $\{e, (i,k)(j,l),(i,l)(j,k), (i,j)(k,l), (i,j,m), (i,j,l,m,k),(i,j,k,m,l)\mid m\not\in\{i,j,k,l\}\}$ & $3n-8$\\
$(i,j,k)$ & $(i,l,j)$ &$2$& $\{e,(i,k,j),(i,k)(j,l),(i,l,k),(j,k,l),(i,l)(j,k),(i,j)(k,l),(i,j,l)\}$ & $8$\\
$(i,j,k)$ & $(i,k,j)$ &$3$& $\{e,(i,l)(j,k),(i,j)(k,l),(i,k)(j,l)\mid l\not\in\{i,j,k\}\}$ & $3n-8$\\
  \bottomrule
\end{tabular*}}
\end{table}

Let $N_\ell(e)$ denote the set of vertices in $CAG_n$ whose distance to the identity vertex $e$ is exactly $\ell$. Thus $N_0(e)=\{e\}$ and $N_1(e)=S$. Also, for any $\gamma\in N_2(e)$, there exists some $\alpha,\beta\in S$ such that $\gamma=\alpha\beta$. By discussing the possible values of $|\mathrm{supp}(\alpha)\cap\mathrm{supp}(\beta)|$ as in the proof of Lemma \ref{lem-3-1}, and using the fact $\gamma\not=e$ and $\gamma\not\in N_1(e)=S$, we may conclude that $\gamma$ is of one of the three types: $\gamma=(i,j,k)(p,q,r)$, $\gamma=(i,j,k,p,q)$ or $\gamma=(i,j)(p,q)$. Moreover, it is easy to see that each permutation having one of these types must lie in $N_2(e)$.  Hence,  $N_2(e)=N_2^1(e)\cup N_2^2(e)\cup N_2^3(e)$, where $N_2^1(e)=\{(i,j,k)(p,q,r)\mid 1\leq i,j,k,p,q,r\leq n, \{i,j,k\}\cap\{p,q,r\}=\emptyset\}$,  $N_2^2(e)=\{(i,j,k,p,q)\mid 1\leq i,j,k,p,q\leq n\}$ and $N_2^3(e)=\{(i,j)(p,q)\mid 1\leq i,j,p,q\leq n, \{i,j\}\cap\{p,q\}=\emptyset\}$. Clearly, $N_2^1(e)=\emptyset$ if $n=5$. The following lemma shows that the automorphisms of $CAG_n$ that fixes the identity vertex $e$ and  all vertices in $N_1(e)$ must be the identity automorphism.

\begin{lem}\label{lem-3-2}
Let $S$ be the complete set of $3$-cycles in $S_n$ ($n\geq 5$), and let $CAG_n=\mathrm{Cay}(A_n,S)$ be the complete alternating group graph. Let $A_e$ denote the set of automorphisms of $CAG_n$ that fixes the identity vertex $e$ and each of its neighbors. Then $A_e=\{1\}$.
\end{lem}
\begin{proof}
Let $g\in A_e$. We shall prove that $g$ fixes each vertex  $\alpha\in N_2(e)=N_2^1(e)\cup N_2^2(e)\cup N_2^3(e)$.

First suppose $\alpha\in N_2^1(e)$. Without loss of generality, we take $\alpha=(1,2,3)(4,5,6)=(4,5,6)(1,2,3)$.  By Lemma \ref{lem-3-1}, there is a unique $4$-cycle containing $e$, $(1,2,3)$ and $(4,5,6)$, namely $(e,(1,2,3),\alpha,(4,5,6))$. Thus $g$ fixes $\alpha$ because $(1,2,3)$, $(4,5,6)\in S=N_1(e)$.

Now suppose $\alpha\in N_2^2(e)$. Without loss of generality, we take $\alpha=(1,2,3,4,5)$. Note that $\alpha\in N((1,2,3))\cap N((1,4,5))=\{e,(1,2,5),(1,2,3,4,5),(1,4,5,2,3),(1,4,3)\}$ (see  Eq. (\ref{equ-1}) or Tab. \ref{tab-1}). Since $g$ fixes $e$ and each vertex in $S=N_1(e)$, we have $\alpha^g=\alpha$ or $\alpha^g=(1,4,5,2,3)$. We claim that the later case never occurs. This is because $\alpha$ has a neighbor $(3,4,5)$ which is fixed by $g$ but $(3,4,5)\not\in N((1,4,5,2,3))$. Hence, each vertex in $N_2^2(e)$ is fixed by $g$.

Finally, suppose $\alpha\in N_2^3(e)$. Without loss of generality, we take $\alpha=(1,2)(3,4)$. From Eq. (\ref{equ-2}) we see that $\alpha\in N((1,2,3))\cap N((1,4,2))=\{e,(1,3,2), (1,3)(2,4),  (1,4,3), (2,3,4),$ $(1,4)(2,3), (1,2)(3,4), (1,2,4)\}$. It follows that $g$ fixes $\{(1,3)(2,4),(1,4)(2,3),(1,2)(3,4)\}$ setwise because $g$ has fixed $e$ and each vertex in $S=N_1(e)$.  We claim that $g$ must fix $\alpha=(1,2)(3,4)$. This is because $\alpha$ has a neighbor $(1,2,3,4,5)\in N_2^2(e)$ which is fixed by $g$ but $(1,2,3,4,5)\not\in N((1,3)(2,4))$ and  $(1,2,3,4,5)\not\in N((1,4)(2,3))$. Hence, each vertex in $N_2^3(e)$ is fixed by $g$.

Therefore, we conclude that $g\in A_e$ fixes all  vertices in $N_2(e)$. Thus we have $(xy)^{g}=xy$ for  $x,y\in S$, which leads to $r_yg r_{y^{-1}}\in A_e$ and then $(xz)^{r_yg r_{y^{-1}}}=xz$ for $x,z\in S$ by the above arguments. It follows that $(xzy)^g=xzy$ and thus $g$ fixes each vertex in $N_3(e)$. Since $CAG_n$ is  connected,   $g$ must fix all vertices of $CAG_n$ by applying induction, and hence $A_e=\{1\}$.

The proof is now complete.
\end{proof}

\begin{lem}\label{lem-3-3}
Let $G_e$ denote the set of automorphisms of $CAG_n=\mathrm{Cay}(A_n,S)$ ($n\geq 5$) that fixes the identity vertex $e$. Let $g\in G_e$, and let $\alpha,\beta$ (not necessarily distinct) be two elements  of $S$. If $\alpha^g=\beta$, then $(\alpha^{-1})^g=\beta^{-1}$.
\end{lem}
\begin{proof}

Without loss of generality, take $\alpha=(1,2,3)$ and $\beta=(i,j,k)$. (Noting that here $\beta=\alpha$, $\beta=\alpha^{-1}$ or $\beta\not\in\{\alpha,\alpha^{-1}\}$.) By the assumption, $(1,2,3)^g=(i,j,k)$, so $r_{(1,2,3)}gr_{(i,j,k)^{-1}}\in G_e$. It follow that
$$
(1,2,3)^{r_{(1,2,3)}gr_{(i,j,k)^{-1}}}=(1,3,2)^g(i,k,j)\in S
$$
because $r_{(1,2,3)}gr_{(i,j,k)^{-1}}$ fixes $S$ setwise. Then there exists $(p,q,r)\in S$ such that $(1,3,2)^g$ $(i,k,j)=(p,q,r)$. Since $g\in G_e$, we have $(1,3,2)^g=(p,q,r)(i,j,k)\in S$. This implies that $|\mathrm{supp}(p,q,r)\cap \mathrm{supp}(i,j,k)|=3$ or $2$. If the former occurs, we have  $(p,q,r)=(i,j,k)$, and so $(1,3,2)^g=(i,k,j)$. If the latter occurs, without loss of generality, say $\{p,q\}=\{j,k\}$ and $r\not\in\{i,j,k\}$, then $(p,q,r)$ must be $(k,j,r)$, and so $(1,3,2)^g=(k,j,r)(i,j,k)=(i,j,r)$. Note that $U=N((1,2,3))\cap N((1,3,2))=\{e, (1,i)(2,3), (1,2)(3,i), (1,3)(2,i)\mid 4\leq i\leq n\}$ and $W=N((1,2,3)^g)\cap N((1,3,2)^g)=N((i,j,k))\cap N((i,j,r))=\{e, (i,k)(j,r),(i,r)(j,k),$ $(i,j)(k,r), (i,j,m), (i,j,r,m,k),(i,j,k,m,r)\mid m\not\in\{i,j,k,r\}\}$ (see Tab. \ref{tab-1}). Since $g$ sends $U$ to $W$, for any fixed $m\not\in\{i,j,k,r\}$, there exists some element $\gamma\in U$ such that $\gamma^g=(i,j,m)\in W$, and so $\gamma=(i,j,m)^{g^{-1}}\in S$, which is impossible since $U$ contains no $3$-cycles. Therefore, if $\alpha^g=\beta$, then $(\alpha^{-1})^g=\beta^{-1}$.

The proof is now complete.
\end{proof}

For $1\leq i<j<k\leq n$, set $\Delta_{i,j,k}=\{(i,j,k),(i,k,j)\}$. Let $\Delta=\{\Delta_{i,j,k}\mid 1\leq i<j<k\leq n\}$. We define $L_n$ as the graph with vertex set  $\Delta$, and with an edge connecting $\Delta_{i,j,k}$ and $\Delta_{p,q,r}$ iff  $\{i,j,k\}\cap\{p,q,r\}= \emptyset$, or equivalently, $\Delta_{i,j,k}\cap\Delta_{p,q,r}=\emptyset$. It is easy to see that $L_n$ is just the  Kneser graph $KG(n, l)$ for $l=3$. By Lemma \ref{lem-2-2}, we know  that $\mathrm{Aut}(KG(n,l))\cong S_n$ for $n\geq 2l+1$. Thus we have

\begin{lem}\label{lem-3-4}
For $n\geq 7$, $\mathrm{Aut}(L_n)\cong S_n$.
\end{lem}

By Lemma \ref{lem-3-2}, $A_e=\{1\}$. Then $G_e=G_e/A_e$ can be viewed as a permutation  group on $S$. By Lemma \ref{lem-3-3}, it is easy to see that $\Delta_{i,j,k}$ is an imprimitive block of $G_e$  on $S$, and furthermore,  $G_e$ acts on $\Delta$. Let $K$ be the kernel of this action. Then $G_e/K$ can be viewed as a permutation group on $\Delta$. In the following lemma, we  show that $G_e/K\leq \mathrm{Aut}(L_n)$.

\begin{lem}\label{lem-3-5}
For $n\geq 5$, $G_e/K\leq \mathrm{Aut}(L_n)$.
\end{lem}
\begin{proof}
Since $G_e/K$ can be viewed as a permutation group on $\Delta$, it suffices to show that $G_e/K$  preserves the adjacency relations in $L_n$.

Let $gK\in G_e/K$. For $\Delta_{i,j,k}\neq \Delta_{p,q,r}\in \Delta$, we see that $\Delta_{i,j,k}\sim\Delta_{p,q,r}$ iff $\{i,j,k\}\cap\{p,q,r\}=\emptyset$ iff  $(i,j,k)$ and $(p,q,r)$ have disjoint support iff $(i,j,k)$ and $(p,q,r)$ commute. Thus it needs to be shown that $(i,j,k)(p,q,r)=(p,q,r)(i,j,k)$ iff $(i,j,k)^{gK}(p,q,r)^{gK}=(p,q,r)^{gK}(i,j,k)^{gK}$. By Lemma \ref{lem-3-1}, $(i,j,k)(p,q,r)$ $=(p,q,r)(i,j,k)$ iff there is a unique $4$-cycle in $CAG_n$ containing $e$, $(i,j,k)$ and $(p,q,r)$, which is the case iff there is a unique $4$-cycle in $CAG_n$
containing $e$, $(i,j,k)^g$ and $(p,q,r)^g$, which is the case iff $(i,j,k)^{g}(p,q,r)^{g}=(p,q,r)^{g}(i,j,k)^{g}$, which is the case iff $(i,j,k)^{gK}(p,q,$ $r)^{gK}=(p,q,r)^{gK}(i,j,k)^{gK}$ because $K$ fixes each block in $\Delta$.

The proof is now complete.
\end{proof}
Remember that $G_e$ acts on $\Delta$, and $K$ is the kernel of this action. The following lemma shows that $K\leq \mathbb{Z}_2$.
\begin{lem}\label{lem-3-6}
Let $K$ be the kernel of  $G_e$ acting on $\Delta$. Then $|K|\leq 2$.
\end{lem}
\begin{proof}
Given $\Delta_{1,2,3}=\{(1,2,3),(1,3,2)\}$, let $K^*$ be the subgroup of $K$ that  fixes $(1,2,3)$ and $(1,3,2)$. Then $[K:K^*]\le 2$ since $|\Delta_{1,2,3}|= 2$. For any fixed $i$ ($4\leq i\leq n$), we first show that $K^*$ fixes $\Delta_{1,2,i}=\{(1,2,i),(1,i,2)\}$ pointwise. In fact, from Tab. \ref{tab-1} we see that $|N((1,2,3))\cap N((1,2,i))|=3n-8\neq 8=|N((1,2,3))\cap N((1,i,2))|$, which implies that  $K^*$ must fix $(1,2,i)$ and $(1,i,2)$ because $K^*$ fixes $\Delta_{1,2,i}=\{(1,2,i),(1,i,2)\}$ setwise. Similarly, for any fixed $j\not\in\{1,i\}$, one can show that $K^*$ fixes $\Delta_{1,i,j}=\{(1,i,j),(1,j,i)\}$ pointwise because $K^*$ has fixed $(1,i,2)$ and $(1,2,i)$. Also, for any fixed $k\not\in\{i,j\}$, one can show that $K^*$ fixes $\Delta_{i,j,k}=\{(i,j,k),(i,k,j)\}$ pointwise because $K^*$ has fixed $(i,j,1)=(1,i,j)$ and $(i,1,j)=(1,j,i)$. By the arbitrariness of $i$, $j$ and $k$, $K^*$ fixes every element in $S$, implying that $K^*\leq A_e=\{1\}$. It follows that $K\leq \mathbb{Z}_2$.

The proof is now complete.
\end{proof}

Combining Lemma \ref{lem-3-4}, Lemma \ref{lem-3-5} and Lemma \ref{lem-3-6}, we have the following corollary.

\begin{cor}\label{cor-3-1}
Let $S$ be the complete set of $3$-cycles in $S_n$ ($n\geq 5$), and let $CAG_n=\mathrm{Cay}(A_n,S)$ be the complete alternating group graph. Then
$$|\mathrm{Aut}(CAG_n)|\leq (n!)^2.$$
\end{cor}
\begin{proof}
Let $G:=\mathrm{Aut}(CAG_n)=\mathrm{Aut}(\mathrm{Cay}(A_n,S))$. With the help of computer software GAP4 \cite{GAP4}, we get $|\mathrm{Aut}(CAG_5)|=14400=(5!)^2$ and $|\mathrm{Aut}(CAG_6)|=518400=(6!)^2$. If $n\geq 7$, we have $G_e/K\leq \mathrm{Aut}(L_n)\cong S_n$ by Lemma \ref{lem-3-4} and Lemma \ref{lem-3-5}. Thus $|G_e|=|G_e/K|\cdot|K|\leq |S_n|\cdot |K|$. By Lemma \ref{lem-3-6}, $|K|\leq 2$. Hence, $|G_e|\leq 2(n!)$. Therefore, $|\mathrm{Aut}(CAG_n)|=|R(A_n)|\cdot|G_e|\leq(n!)^2$, and our result follows.
\end{proof}
By Theorem \ref{thm-2-2} and Corollary \ref{cor-3-1}, we obtain the main result of this paper immediately.
\begin{thm}\label{thm-3-1}
Let $S$ be the complete set of $3$-cycles in $S_n$ ($n\geq 5$), and let $CAG_n=\mathrm{Cay}(A_n,S)$ be the complete alternating group graph. Then the automorphism group of $CAG_n$ is
$$\mathrm{Aut}(CAG_n)=(R(A_n)\rtimes \mathrm{Inn}(S_n))\rtimes \mathbb{Z}_2\cong (A_n\rtimes S_n)\rtimes \mathbb{Z}_2,$$
where $R(A_n)$ is the right regular representation of $A_n$, $\mathrm{Inn}(S_n)$ is the inner automorphism group of $S_n$, and $\mathbb{Z}_2=\langle h\rangle$, and $h$ is the map $\alpha\mapsto\alpha^{-1}$ ($\forall\alpha\in A_n$).
\end{thm}
\begin{remark}\label{rem-5}
\emph{For $n=3$, we see that $CAG_3\cong K_3$, and thus $\mathrm{Aut}(CAG_3)\cong S_3$. For $n=4$, by using the  computer software GAP4 \cite{GAP4}, we obtain $|\mathrm{Aut}(CAG_4)|$ $=82944=144\cdot(4!)^2$, and further, $\mathrm{Aut}(CAG_4)\cong((((((A_4\times A_4)\rtimes \mathbb{Z}_2) \times A_4) \rtimes \mathbb{Z}_2) \rtimes  \mathbb{Z}_3) \rtimes  \mathbb{Z}_2) \rtimes  \mathbb{Z}_2$. Clearly, $\mathrm{Aut}(CAG_4)$ has a different form in comparison with  $\mathrm{Aut}(CAG_n)$ for $n\geq 5$.}
\end{remark}

\section{Conclusion}
The non-normality of the  complete alternating group graph $CAG_n$ has been proved, and the automorphism group of $CAG_n$ has been obtained. For $n\geq 5$, $n\neq 6$ and $4\leq k\leq n$, let $T$ be the set of all $k$-cycles in $S_n$. In Remarks \ref{rem-1} and \ref{rem-3}, we have seen that $\mathrm{Aut}(\mathrm{Cay}(\Gamma,T))\supseteq (R(\Gamma)\rtimes \mathrm{Inn}(S_n))\rtimes \mathbb{Z}_2\cong (\Gamma\rtimes S_n)\rtimes \mathbb{Z}_2$,  where $\Gamma=A_n$ or $S_n$ due to  $k$ is odd or even. Naturally, we have the following problem.
\begin{prob}\label{prob-1}
Let $T$ and $\Gamma$ be defined as above. Is $\mathrm{Aut}(\mathrm{Cay}(\Gamma,T))$ exactly equal to $(R(\Gamma)\rtimes \mathrm{Inn}(S_n))\rtimes \mathbb{Z}_2\cong (\Gamma\rtimes S_n)\rtimes \mathbb{Z}_2$?
\end{prob}

\section*{Acknowledgments}

We are grateful to the anonymous referees for their useful and constructive comments, which have considerably improved the readability of this paper.

\end{document}